\documentclass[a4,11pt]{amsart}
\usepackage{amssymb,latexsym}
\usepackage{color}
\usepackage[left=35mm, right=35mm, top=30mm, bottom=30mm, includefoot, includehead]{geometry}
\usepackage{xspace}
\usepackage{color}
\usepackage{graphicx}

\theoremstyle{plain}
\newtheorem{theorem}{Theorem}[section]
\newtheorem*{theorem*}{Main Theorem}
\newtheorem{proposition}[theorem]{Proposition}

\newtheorem{lemma}[theorem]{Lemma}
\newtheorem{conjecture}[theorem]{Conjecture}
\newtheorem{question}[theorem]{Question}

\theoremstyle{definition}
\newtheorem{definition}[theorem]{Definition}

\newtheorem{remark}[theorem]{Remark}
\newtheorem{example}[theorem]{Example}

\newcommand{\enm}[1]{\ensuremath{#1}}       
\newcommand{\op}[1]{\operatorname{#1}}
\newcommand{\cal}[1]{\mathcal{#1}}

\newcommand{\NN}{\enm{\mathbb{N}}}
\newcommand{\RR}{\enm{\mathbb{R}}}

\renewcommand{\AA}{\enm{\mathbb{A}}}
\newcommand{\PP}{\enm{\mathbb{P}}}

\newcommand{\Ff}{\enm{\cal{F}}}

\newcommand{\Ii}{\enm{\cal{I}}}

\newcommand{\Mm}{\enm{\cal{M}}}

\newcommand{\Oo}{\enm{\cal{O}}}

\newcommand{\Ss}{\enm{\cal{S}}}

\newcommand{\Ww}{\enm{\cal{W}}}

\renewcommand{\phi}{\varphi}
\renewcommand{\theta}{\vartheta}
\renewcommand{\epsilon}{\varepsilon}

\newcommand{\supp}{\op{Supp}}

\newcommand{\red}{\mathrm{red}}
\newcommand{\reg}{\mathrm{reg}}
\newcommand{\Sing}{\mathrm{Sing}}

\renewcommand{\to}[1][]{\xrightarrow{\ #1\ }}

\newcommand{\old}[1]{}


\begin{document}

\title[]{Injective linear series of algebraic curves on quadrics}

\author{Edoardo Ballico and Emanuele Ventura}

\address{Universit\`a di Trento, 38123 Povo (TN), Italy}
\email{edoardo.ballico@unitn.it}

\address{Universit\"{a}t Bern, Mathematisches Institut, Sidlerstrasse 5, 3012 Bern, Switzerland}
\email{emanueleventura.sw@gmail.com, emanuele.ventura@math.unibe.ch}

\maketitle

\begin{abstract}
We study linear series on curves inducing injective morphisms to projective space, using zero-dimensional schemes and
cohomological vanishings.  Albeit projections of curves and their singularities are of central importance in algebraic geometry, basic problems still remain unsolved. In this note, we study cuspidal projections of space curves lying on irreducible quadrics (in arbitrary characteristic). 
\end{abstract}

\section{Introduction}

Projections and singularities of curves are of central importance in algebraic geometry. The projective geometry of singular curves is a delightful chapter of classical algebraic geometry that remains active even up to this date: many questions await to be settled, and in turn they inspire the introduction of
tools entailing deformation theory, zero-dimensional schemes, and combinatorics, among other techniques. 

A natural direction of research is the classification of singularities that may arise on a curve $X$, in some specific ranges of the numerical invariants attached to $X$.  An approach to this classification issue, relying on osculating spaces and combinators of semigroups of valuations, has been recently employed in \cite{BIV}. Other classification results, leveraging the structure of free resolutions of certain ideals, have been achieved in \cite{CKPU}. 

In this note, we employ to some extent the standpoint of Greuel, Lossen, and Shustin \cite{GLS}, using the geometry of zero-dimensional schemes and the cohomology of their ideal sheaves, to study {\it cuspidal} (or unibranch) singularities. These types of singular points are usually related to some {\it tangency} conditions and so carry interesting geometric information about the curve. 

Let $k$ be an algebraically closed field. Let $X$ be a complete smooth curve of genus $g$ over $k$, i.e., an integral scheme of dimension one, smooth and proper over $k$. Every such $X$ is projective and can be embedded in projective $3$-space, independently of the characteristic of $k$.  A natural question to wonder about is whether every $X$ admits a projection to $\PP^2$ with only cuspidal singularities, i.e. $X$ admits a cuspidal projection. 
Ferrand \cite{ferrand} showed that, when $\mathrm{char}(k)>0$, if $X$ admits a cuspidal projection then $X$ is a set-theoretic complete intersection. 

Thereafter, back to characteristic zero, Piene \cite{Piene} proved that every $X\subset \PP^3$ of degree $d$ and genus $g$, when $\binom{d-1}{2}-g\leq 3$, admits a cuspidal projection. However, a general canonical curve $X\subset \PP^3$ of genus $4$ does not admit a cuspidal projection \cite[Theorem 2]{Piene}. (Note that the latter curve is a complete intersection, so Ferrand's result does not hold in $\mathrm{char}(k)=0$.) Sacchiero \cite{GSac} showed that, in the range $\binom{d-1}{2}-g\geq 4$, a generic projection of an $X\subset \PP^n$ ($n\geq 4$), without inflection points and without hyperosculating planes, is not cuspidal. 

For cuspidal curves, a classical open problem is to determine the maximum number of cusps realizable on a plane 
curve of degree $d$; recent asymptotic results were proven by Calabri, Paccagnan, and Stagnaro \cite{cps}. Interestingly, Koras and Palka \cite{kp19} showed that complex plane rational cuspidal curves possess at most four singular points. 

Cuspidal projections play an important role in the theory of $X$-ranks. Let $X\subset \PP^3$ be a smooth curve and $p\in \PP^3\setminus X$. Then the $X$-rank of $p$ satisfies $\mathrm{rk}_X(p)>3$ if and only if the projection of $X$ away from $p$ is cuspidal. See \cite{BBV} for more results in this direction. 

What originally triggered this work has a topological source. Motivated by the study of regular topological maps \cite{BJJM} in the case of smooth curves, Micha\l{}ek posed the problem \cite{MM}: 

\begin{question}\label{q1}
For any $X$ over a field $k$ as above, does there exist an injective morphism $\phi: X\rightarrow \PP^2$? 
\end{question}

\noindent This is also studied for other projective varieties by G\"{o}rlach \cite{G19}. The curve $\phi(X)$ is then an integral plane curve possibly with only cuspidal singularities. The map $\phi: X\to \phi(X)$ is a closed bijection and so a homeomorphism in Zariski topology; cuspidal projections of $X$ to $\PP^2$ are instances of injective morphisms. 

The aim of this note is to study base-point free (not necessarily complete) two-dimensional linear series $g^2_d$ on some smooth algebraic curves $X$ inducing separable and injective morphisms to $\mathbb P^2$; we call these linear series {\it injective}. In this article, these are usually constructed by cuspidal projections of $X$ to $\PP^2$ with the help of zero-dimensional schemes and their cohomology, which will let us give positive instances to Question \ref{q1}.  

In this context, we propose the following conjecture:

\begin{conjecture}\label{i1}
For large $g$, a very general smooth curve of genus $g$ has no injective linear series $g^2_d$.
\end{conjecture}

We spell out the meaning of  ``very general'' in the statement of Conjecture \ref{i1}. Let $Y$ be an integral quasi-projective variety. Fix a property $\wp$ that a point $p\in Y$ may satisfy. We say that $\wp$ is true for a very general point of $Y$ if the set of all $p\in Y$ for which $p$ fails $\wp$ is contained in a countable family of proper subvarieties of $Y$. In Conjecture \ref{i1}, the generality is applied to the moduli scheme $\Mm _g$ ($g\ge 2$) of all smooth curves of genus $g$ (over some fixed algebraically closed field). We guess that more should be true: for large $g$ and for every positive integer $d$, the set of all $X\in \Mm _g$ with an injective $g^2_d$ sits inside a proper subvariety of $\Mm _g$. 

We now clarify the meaning of ``large $g$'' in the statement of Conjecture \ref{i1}: this refers to the existence of an integer $g_0(k)$ (depending on the fixed algebraically closed ground field $k$) such that, for all $g \geq g_0(k)$, a very general curve of genus $g$ has no injective linear series $g^2_d$.

We believe it would be interesting to have partial results on Conjecture \ref{i1}, for non-complete $g^2_d$, i.e., for $g^2_d$
inducing a non-degenerate injective map $j: X\to \PP^n, n>2$, composed with a linear projection. This is the setting of Piene \cite{Piene} and Sacchiero \cite{GSac}, except that they require $j$ to be an embedding. Furthermore, we ask the following

\begin{question}\label{q2}
Let $X$ be a smooth curve of genus $g$. Are there infinitely many integers $d$ such that $X$ has injective $g^2_d$?
\end{question}

Even if Conjecture \ref{i1} fails, we ask whether, for all sufficiently large $g$, there exists an $X\in \Mm _g$ with no injective $g^2_d$. In such a case, one may still wonder whether, for infinitely many genera $g$, a very general curve of genus $g$ has no injective $g^2_d$.

Most of our results arise from looking at the quadric $Q =\PP^1\times \PP^1$. Let $X$ be a smooth projective curve of genus $g$. By the universal property of the fibered product of schemes, giving a morphism $f: X\to \PP^1\times \PP^1$ is equivalent to prescribing two morphisms $u_i: X\to \PP^1$, $i=1,2$, i.e., two base-point free linear series $g^1_{d_1}$ and $g^1_{d_2}$, where
$d_1 = \deg u_1^\ast (\Oo _{\PP^1}(1))$ and $d_2 = \deg u_2^\ast (\Oo _{\PP^1}(1))$. (Here we assume $d_1, d_2\ne 0$, as otherwise $f(X)$ is contained in a line of $\PP^3$.) The morphism $f$ is birational onto its image if and only if there is no $4$-tuple $(D,h,v_1,v_2)$, where $D$ is a smooth projective curve, $h: X\to D$ is a finite morphism with $\deg (h)>1$, $v_i: D\to \PP^1$, $i=1,2$, are morphisms and $u_i = v_i\circ h$, $i=1,2$. In classical terminology, $f$ is birational onto its image if and only if $g^1_{d_1}$ and $g^1_{d_2}$ are not composed with the same involution. The pair $(d_1,d_2)$ is defined to be the bidegree of $f$.
\begin{question}
For which $(X,d_1,d_2)$, is there an $f$ of bidegree $(d_1,d_2)$ that is injective and separable? For which $X$, are there infinitely many $(d_1,d_2)$ such that there is an injective and separable $f$? For which pair $(X,d_1)$, with $d_1>1$, are there infinitely many integers $d_2$, such that there exists an injective and separable $f: X\to \PP^1\times \PP^1$ of bidegree $(d_1,d_2)$?
\end{question}

\noindent{\bf Main result.} Let $X$ be a complete smooth curve of genus $g$ over $k$. An {\it injective linear series} on $X$ is a (not necessarily complete) series $g_d^2$ inducing a separable and injective morphism $\phi: X\rightarrow \PP^2$. We organize injective linear series into natural {\it types}: 

\begin{definition}\label{deftypes}
Let $L = \phi^\ast (\Oo_{\PP^2}(1))$ be the line bundle on $X$ associated to $g^2_d$, i.e. for a divisor $D\in g^2_d$ one has $L = \Oo_X(D)$.
An injective $g^2_d$ on $X$ has one of the following {\it types}: 
\begin{enumerate}
\item[({\bf I})] a complete $g^2_d$, i.e. $h^0(L)=3$; 
\item[({\bf II})] an incomplete $g^2_d$, i.e. $h^0(L)\ge 4$, with $L$ very ample line bundle;
\item[({\bf III})] an incomplete $g^2_d$ with $L$ not very ample.
\end{enumerate}
\end{definition}
\noindent (See Proposition \ref{rmktypes} for some geometric remarks about them.) This is our main result: 

\begin{theorem*}[{\bf Theorems \ref{pp1} and \ref{bb6}}]
Let $k$ be an algebraically closed field of arbitrary characteristic and let $d_2\geq d_1\geq 1$. Then there exists a smooth genus $g$ curve with an injective $g^2_{d_1+d_2}$ of {\it type} {\bf II} with $g= d_1d_2-d_1-d_2+1$. 

Let $k$ be an algebraically closed field of characteristic zero. Let $d_2\ge d_1\ge 16$ and $h>0$ such that $3h+2\le \binom{d_1-1}{2}$. Fix an integer $\kappa$ such that $0< \kappa \le 2h$ and set $g = d_1d_2-d_1-d_2+1-\kappa$. Then there exists a smooth genus $g$ curve with an injective $g^2_{d_1+d_2}$ of type {\bf III}. 
\end{theorem*}

\noindent {\bf More contributions and structure of the paper.}
In \S \ref{first two types}, we work in $\mathrm{char}(k)=0$. We first record some observations about the geometry behind 
the types of linear series in Proposition \ref{rmktypes}. In Theorem \ref{g0}, we use an existence result of Barkats \cite{bar} to show that in every genus there exist curves equipped with {\it type} {\bf I} injective linear series. We recall an important (existence and smoothness) result of Greuel, Lossen, and Shustin about the variety $V(d,\kappa)$, parametrizing plane curves of given degree $d$ and with $\kappa$ ordinary cusps as their only singularities; see Theorem \ref{g1}. In Remark \ref{glscurves}, we point out that the curves from this result are essentially different from those arising in our Theorems \ref{pp1} and \ref{bb6}. 

In \S \ref{hyperelliptic}, we offer a study of injective linear series on hyperelliptic curves. Surprisingly, this material seems new and it is interesting on its own right. 
We explicitly describe $2g+2$ families of $\infty^1$-many injective linear series $g^2_{g+3}$ of degree $g+3$, see Theorem \ref{a1}. These families are in correspondence to the Weierstrass points of a given hyperelliptic curve along with a double cover of $\PP^1$. Although the statement of Theorem \ref{a1} is in characteristic zero, we point out that the result holds in any characteristic with suitable modifications of the arguments; this is observed in Remark  \ref{anychar}. 
Proposition \ref{g+3inj} and Proposition \ref{g+3injbis} give a characterization of injective (non-special) linear series of degree $g+3$. Proposition \ref{a3} provides a description of injective linear series of degree $g+2$. 

In \S \ref{quadrics}, we study more closely space curves lying on irreducible quadrics in $\PP^3$. With the help of zero-dimensional schemes in arbitrary $\mathrm{char}(k)$, we prove Theorem \ref{b1} and Theorem \ref{b4.1}, showing the existence of smooth curves on smooth quadrics and cones admitting cuspidal projections; these two results yield Theorem \ref{pp1}. 

In order to establish Theorem \ref{bb6}, we employ results of Ro\'{e} and of Greuel, Lossen, and Shustin, extending them to smooth quadrics $Q\subset \PP^3$; this is achieved in Lemmas \ref{bb2}, \ref{bb3}, \ref{bb4}, and \ref{bb5}. 

In \S\ref{innercusp}, we introduce two sets $\mathcal A$ and $\mathcal B$, naturally attached to a (inner smooth) cuspidal projection of a curve from a point. Theorem \ref{f2} and Theorem \ref{f3} provide a characterization of curves in $\PP^3$, lying on smooth quadrics and quadric cones, with only cuspidal singularities in terms of $\mathcal A$ and $\mathcal B$.

\vspace{4mm}
\noindent {\bf Acknowledgements.}
The first author is partially supported by MIUR and GNSAGA of INdAM (Italy). The second author is supported 
by the grant NWO Den Haag no. 38-573 of Jan Draisma.

\section{Types and Castelnuovo's bound}\label{first two types}

In this section, we work in $\mathrm{char}(k)=0$. We start formulating a proposition recording some geometric remarks behind the types of injective linear series: 
\begin{proposition}\label{rmktypes}
Keep the notation from Definition \ref{deftypes}. Then the following hold: 
\begin{enumerate}
\item[(i)] If $g^2_d$ has {\it type} {\bf II}, then there exists $g^3_d\subseteq |L|$ with $g^3_d\supset g^2_d$; the latter $g^3_d$ induces an embedding $j(X)\to \PP^3$ such that the morphism $\phi$ is the composition of $j$ with a linear projection of $\PP^3$ from a point of $\PP^3\setminus j(X)$ \cite[p. 110, parts 4) and
5)]{Piene}. 

\item[(ii)] If $g^2_d$ is of {\it type} {\bf III}, then for any $3\le s\le h^0(L)-1$, any
$g^s_d\subseteq |L|$ containing $g^2_d$ is base-point free and it induces a morphism $u: X\to \PP^s$ birational onto its image,
but not an embedding; the curve $\phi(X)$ is obtained from $u(X)$ by a linear projection from an $(s-3)$-dimensional linear subspace of
$\PP^s$ not intersecting $u(X)$.
\end{enumerate}
\end{proposition}

\begin{remark}\label{halphencastelnuovo}
Since a cuspidal $g^2_d$ has no base points and it induces an injective morphism, a necessary condition for the existence of a $g^2_d$ of {\it type} {\bf II} or {\bf III} on $X$ is that $X$ has a degree $d$ and genus $g=p_a(X)$ non-degenerate birational model in $\PP^3$. 

Fix integers $d, g$ such that $g\ge 0$ and $d\ge 3$. Define: 
\[
\pi(d,3) = m(m-1) + m\epsilon, \mbox{ where } \epsilon\in \lbrace 0,1\rbrace \mbox{ and } d=2m+1+\epsilon.
\] 
Halphen and later Castelnuovo proved that if there is a non-degenerate space curve $X\subset \PP^3$ of degree $d$ and arithmetic genus $g$, then $g\le
\pi(d,3)$ \cite[Theorem 3.7]{he}.  Not all the possible integers $g\le \pi(d,3)$ arises as arithmetic genera, even allowing singular curves. 
\end{remark}

Thus the question of existence of the possible pairs $(g,d)$ is natural: 

\begin{question}\label{qq1}
For which $g$ and $d$, does there exist a smooth curve of genus $g$ with an injective $g^2_d$ of types {\bf I} or {\bf II}?
\end{question}

\begin{remark}
Question \ref{qq1} was partially answered by  Ephraim and Kulkarni, who proved that for each genus $g\ge 0$ and each integer $d>2g$ there exists a curve of genus $g$ with a type {\bf II} injective $g^2_d$ \cite[Corollary 3.9]{EK}.
\end{remark}

\begin{theorem}\label{g0}
For each integer $g\ge 0$ there exists a smooth curve of genus $g$ with a type {\bf I} injective linear series.
\end{theorem}

\begin{proof}
Let $d$ be the minimal integer $\ge 2$ such that $g\le (d-1)(d-2)/2$. If $g = (d-1)(d-2)/2$ it is sufficient to take as $X$ a
smooth degree $d$ plane curve. Thus we may assume $d\ge 4$ and $(d-2)(d-3)/2 < g < (d-1)(d-2)/2$. Define: 
\[
\kappa = (d-1)(d-2)/2 - g.
\]
Note that $1\le \kappa \le d-3$ and hence $5\kappa \le 5(d-3)\leq (d+2)(d+1)/2-d-1$. Since $5\kappa \le (d+2)(d+1)/2-d-1$, there exists an integral plane curve 
$Y\subset \PP^2$ of degree $d$ with $\kappa$ ordinary cusps as its only singularities by the work of Barkats \cite{bar} for $d\geq 5$ (for $d=4$ there exists a plane curve with a unique cusp, i.e. the projection of a curve of genus $g=2$). 

Thus there exists an injective $g^2_d$ on a smooth genus $g$ curve. Now we discuss
why we may find some complete $g^2_d$. If $g=0$ (resp. $g=1$, resp. $g=3$) we take a smooth plane curve of degree $2$ (resp. $3$, resp. $4$).
If $g=2$, we take a degree $4$ plane curve with an ordinary cusp as its unique singular point; this linear series $g^2_4$ is complete,
because no genus $2$ curve has a $g^3_4$ (see e.g. \cite[Corollary IV.6.2]{Hart}). 

Assume $d\ge 5$ and $1\le \kappa\le d-3$. Let $Y\subset \PP^2$ be any integral plane curve with exactly $\kappa$ ordinary cusps as singularities.
Let $u: X\to Y$ be the normalization map and let $\phi: X\to \PP^2$ denote the composition of $u$ with the inclusion $Y\subset \PP^2$. Define $L = \phi^\ast (\Oo _{\PP^2}(1))$. 

We need to prove that $h^0(L) =2$. Letting $g = p_a(X)$, by Riemann-Roch, it is enough to show that $2-h^1(L) =d+1-g$, i.e. $h^1(L) =g+1-d$. Let $S = \mathrm{Sing}(Y)$. Since each singular point of of $X$ is an ordinary cusp, the classical Pl\"{u}cker formulas affirm that $g=(d-1)(d-2)/2 - \kappa$ \cite[page 280]{GH}. The equality is in fact derived from the cohomological equality $H^0(X,\omega _X) =\phi^\ast (H^0(\PP^2,\Ii _S(d-3))$. By Serre duality, $h^1(L) = h^0(\omega _X\otimes L^\ast)$. Since $H^0(X,\omega _X) =\phi^\ast (H^0(\PP^2,\Ii _S(d-3))$, we have 
$h^1(L) =g+1-d$ if and only if $h^1(\PP^2,\Ii _S(d-4)) =0$. For a general $S$ we have $h^1(\PP^2,\Ii _S(d-4)) =0$, because $\#S= \kappa \le (d-3)(d-2)/2 =h^0(\Oo _{\PP^2}(d-4))$.
\end{proof}

Besides Barkats' results \cite{bar} on the existence of curves with prescribed singularities, another more recent result is due to Greuel, Lossen, and Shustin \cite[Corollary 2.4]{gls3}. (Here we explicitly state a special case of the latter for cuspidal curves.)

\begin{theorem}[{\bf Greuel, Lossen, Shustin}]\label{g1}
Let $V(d,\kappa)$ denote the set parametrizing all irreducible plane curves $X\subset \PP^2$ with $\deg (X)=d$ and $\kappa$ ordinary
cusps as its only singularities. Assume $9\kappa < d^2+6d+8$. If $V(d,\kappa)\ne \emptyset$, then $V(d,\kappa)$ is smooth of pure
dimension $(d^2+3d)/2 -2\kappa$.
\end{theorem}

\begin{remark}\label{glscurves}
Irreducibility of $V(d,\kappa)\ne \emptyset$ requires $18\kappa < d^2$, see \cite[Corollary 3.2]{gls3}. The non-emptyness $V(d,\kappa)\ne \emptyset$ was proven over $\RR$ by Shustin \cite[Theorem 3.3]{sh1}. For all positive integers $d>0$, let $\kappa (d)$ denote the maximal integer such that for all $0\le \kappa \le \kappa(d)$ there exists $Y\in V(d,\kappa)$ defined over $\RR$ with 
$\mathrm{Sing}(Y)\subset \PP^2(\RR)$. One sees $\kappa(1) =\kappa(2)=0$, $\kappa(4)=3$, $\kappa(5)=5$ and
$\kappa(6)=7$ \cite[page 851]{sh1}. More generally: for all $d\ge 7$, one has $\kappa (d)\ge (d^2-3d+4)/4$ if $d\equiv 0,3\pmod{4}$ and $\kappa(d)
\ge (d^2-3d+2)/4$ if $d\equiv 1,2\pmod{4}$ \cite[Theorem 3]{sh1}. Concerning the smoothness of varieties parametrizing plane curves (not necessarily with only unibranch singularities) see the results in \cite[\S 4.3]{GLS}. The beautiful book \cite{GLS} contains an extensive bibliography about this venerable subject; in particular, more information about singular curves on more surfaces other than the projective plane are discussed.
\end{remark} 

\begin{remark}
The plane curves $X$ realized in Theorem \ref{g1} do not come from injective linear series $g^2_d$ of {\it type} {\bf II} or {\bf III} whenever 
\[
p_a(X) = (d-1)(d-2)/2 - \kappa > \pi(d,3). 
\]
Hence such injective linear series $g^2_d$ must be complete and so of {\it type} {\bf I}. Therefore, the curves from Theorem \ref{g1} are different from those arising in our Theorems \ref{pp1} and \ref{bb6}. 
\end{remark}

\section{Injective linear series on hyperelliptic curves}\label{hyperelliptic}

From the classical analytic definition of complex hyperelliptic curves, i.e., as the Riemann surface of the algebraic function $y=\sqrt{(x-a_1)\cdots (x-a_{2g+2})}$ \cite[\S 2]{ACGH}, it is clear that they admit a cuspidal model in $\PP^2$ (the cuspidal point being at infinity) and therefore they carry an injective linear series
from their very definition. However, we describe other natural injective linear series on any hyperelliptic curve in any characteristic. Their construction is highlighted in the course of the proof of the next Theorem \ref{a1}; to our knowledge this is new and interesting on its own right. The result is stated in proven first in characteristic zero, and in the subsequent Remark \ref{anychar} the general case is treated.

\begin{theorem}\label{a1}
Let $X$ be a smooth hyperelliptic curve of genus $g\ge 2$ over a field $k$ with $\textnormal{char}(k)=0$. Then there is a base-point free $g^2_{g+3}$ inducing an injective morphism $\phi: X \to \PP^2$. The image $\phi(X)$ has exactly two singular points: one ordinary cusp and one unibranch singularity. Moreover, $X$ has $\infty ^1$-many such $g^2_{g+3}$, each of them being a sublinear series of a different complete and very ample $g^3_{g+3}$; $X$ has $2g+2$ such one-dimensional families of $g^2_{g+3}$ and $g^3_{g+3}$.
\end{theorem}

\begin{proof}
Let $\Ww$ be the set of Weierstrass points of $X$, i.e., the support of the ramification divisor of the $2:1$ cover $u_2: X\rightarrow \mathbb P^1$, induced by the linear series $g^1_2$ on $X$. (This exists on $X$, as it is hyperelliptic.) 

Fix a point $o\in \Ww$. Then $o\in \Ww$ if and only if $2o\in g^1_2$ by definition of ramification divisor. 
Thus $2 = h^0(g^1_2) = h^0(\Oo _X(2o))$. For each $p\in X\setminus \Ww$, set $N_p:= \Oo _X(2o+(g+1)p)$. We now split the proof into four claims. 

\begin{quote}
\noindent {\it Claim 1:} For a general $p\in X\setminus \Ww$, we have $h^0(\Oo _X(gp)) =1$, $h^0(\Oo _X(g+1)p)) =2$ and $h^0(N_p) =4$. \\

\noindent \emph{Proof of Claim 1:} By Riemann-Roch, we have $h^0(\Oo _X(gp)) =1+h^1(\Oo _X(gp))$, $h^0(\Oo _X(g+1)p)) =2+h^1(\Oo _X((g+1)p))$ and $h^0(N_p) =4+h^1(N_p)$.  Notice that $h^1(\Oo _X(2o+(g+1)p))\le h^1(\Oo _X((g+1)p)) \le h^1(\Oo _X(gp))$. Hence, in order to finish the proof of {\it Claim 1}, it is sufficient to prove that $h^1(\Oo _X(gp)) =0$. Indeed, under this assumption,  as $\textnormal{char}(k) = 0$, for any invertible sheaf $\mathcal N$ on $X$ and for a general $p\in X$, one has $h^0(\mathcal N(-tp)) = \max\lbrace 0, h^0(\mathcal N) - t \rbrace$, for any positive integer $t$. Letting $\mathcal N = \omega_X$ and $t=g$, 
we obtain $h^0(\omega_X(-gp)) = 0$. Finally, Serre duality gives $h^1(\Oo_X(gp)) = 0$. 
\end{quote}

\begin{quote}
\noindent {\it Claim 2:} For a general $p\in X$, $N_p$ is very ample.\\

\noindent \emph{Proof of Claim 2:} The base locus of $|\Oo _X((g+1)p)|$ is contained in $\{p\}$. Since $h^0(\Oo _X(gp)) < h^0(\Oo _X((g+1)p))$ by {\it Claim 1}, it follows that $p$ is not a base point of $|\Oo _X((g+1)p)|$, as subtracting a base point from a divisor does not decrease dimension of global sections. Thus $\Oo_X((g+1)p)$ is base-point free and so its linear series $|\Oo _X((g+1)p)|$ defines a degree $g+1$ morphism $u_1: X\to \PP^1$. As above, denote $u_2: X\to \PP^1$ the degree $2$ cover of $\PP^1$ induced by $g^1_2$ on $X$. The pair $(u_1,u_2)$ induces a morphism $w: X\to \PP^1\times \PP^1$. Now since $u_2$ is a degree $2$ cover, either $w$ is birational onto its image or it factors through $u_2$. The latter case is not possible, because $u_1$ cannot factor through $u_2$. Indeed, for the sake of contradiction, assume that $u_1$ factors through $u_2$. Then $\Oo_X((g+1)p)$ would be isomorphic to the invertible sheaf $(g^1_2)^{\otimes (g+1)/2}$. Since the dimension of the linear series of the latter is $(g+1)$ and $g\geq 2$, this isomorphism implies $h^0(\Oo_X((g+1)p)) > 2$, which contradicts {\it Claim 1} above. Hence $w$ is birational onto its image.
\end{quote}

Recall that the canonical sheaf of $\PP^1\times \PP^1$ is $\omega _{\PP^1\times \PP^1} \cong \Oo _{\PP^1\times \PP^1}(-2,-2)$. For $D\in |\Oo _{\PP^1\times \PP^1}(2,g+1)|$, by adjunction, $\omega_D \cong \omega_{\PP^1\times \PP^1} \otimes \Oo_{\PP^1\times \PP^1}(2,g+1)$. (For singular $D$, replace $\omega_D$ with the dualizing sheaf $\omega^{\circ}_D$ and every later statement holds as well.) This implies that the arithmetic genus of each $D\in |\Oo _{\PP^1\times \PP^1}(2,g+1)|$ is $p_a(D) =g$. 
 
Since $w$ is birational onto its image, $w(X)$ has bidegree $(2,g+1)$ and so $w(X)\in |\Oo _{\PP^1\times \PP^1}(2,g+1)|$. Since $w(X)$ has arithmetic genus $g$, the morphism $w$ is an embedding. 

The linear series $|\Oo _{\PP^1\times \PP^1}(1,1)|$ embeds $\PP^1\times \PP^1$ as a quadric surface in $\PP^3$. Call $f$ the composition of $w$ and the inclusion $\PP^1\times \PP^1\hookrightarrow \PP^3$. By construction, $f$ is the map induced by $|N_p|$; hence $N_p$ is very ample.

Take a very ample divisor of the form $N_p$ with associated embedding $f: X\to\PP^3$ and, as in the proof of {\it Claim 2}, regard $f(X) \subset \PP^1\times \PP^1$ as a divisor of bidegree $(2,g+1)$ on the quadric surface $\PP^1\times \PP^1$. Let $q\in \PP^1\times \PP^1$ be the point $(u_1({p}), u_2(o))$. For a general $p$, we may assume $u_1({p})\ne u_2(o)$. With this assumption, we show the next

\begin{quote}
\noindent {\it Claim 3:} We have $q\notin f(X)$.\\

\noindent \emph{Proof of Claim 3:} Assume $q\in f(X)$. The line $L_1:= \PP^1\times \{u_2(o)\}$ is tangent to $f(X)$ at $f(o)$
because it intersects $f(X)$ at $f(o)$  with multiplicity two. The line $L_2:= \{u_1(p)\}\times \PP^1$ is tangent to $f(X)$ at $u_2({p})$, because it intersects
$f(X)$ at $f(p)$ with multiplicity $g+1$. Since $L_1, L_2$ are lines in a different ruling of the quadric surface $\PP^1\times
\PP^1$, $L_1\ne L_2$ and $L_1\cap L_2$ is a single point. Note that $\{q\} = L_1\cap L_2$. Since $f(X)$ is smooth at $q$, it
has a unique tangent line at $q$. Thus $L_1= L_2$, which is a contradiction.
\end{quote}

Let $\pi_q: \PP^3\setminus \{q\} \to \PP^2$ denote the linear projection from $q$. Since $q\notin f(X)$, $\pi _{q|f(X)}$ induces a morphism $\phi: X\to \PP^2$. Since $\deg (u_2)=2$ and $u_1$ does not factor through $u_2$, $\phi$ is birational onto its image. To conclude the proof of the theorem, it is sufficient to prove the following claim.

\begin{quote}
\noindent {\it Claim 4:} The morphism $\phi$ is injective, $o$ and $p$ are the only ramification points of $\phi$, and $\phi(o)$ is an ordinary cusp of $\phi(X)$.\\

\noindent {\it Proof of Claim 4:} Since $f$ is an embedding and $\phi$ is induced by the linear projection from $q\in \PP^1\times \PP^1\setminus f(X)$, it is sufficient to prove that $|L\cap f(X)|\le 1$ for each line $L\subset \PP^3$ containing $q$. Fix a line $L\subset \PP^3$ such that $q\in L$ and $\deg (f(X)\cap L) \ge 2$. Since $q\notin f(X)$ and $f(X)\subset \PP^1\times \PP^1$, B\'{e}zout's theorem gives $L\subset \PP^1\times \PP^1$. Thus $L$ is one of the two lines of the smooth quadric $\PP^1\times \PP^1$ passing through $q$. One of these lines meet $f(X)$ only at $f(o)$ (with multiplicity two), whereas the other one meets $f(X)$ only at $f({p})$ (with multiplicity $g+1$). In both cases, the set-theoretic intersection $L\cap f(X)$ consists only of one point. Thus $\phi$ is injective. Moreover, by the discussion above, $o$ and $p$ are the only ramification points. Since $\phi^{-1}(o)$ is a curvilinear double point, $f(o)$ is a double point with one branch and so an ordinary cusp. 
\end{quote}

In conclusion, $X$ has $\infty ^1$-many $g^2_{g+3}$ base-point free linear series, corresponding to the morphism $\phi: X\rightarrow \PP^2$, as $p$ varies in $X\setminus \Ww$. (The degree of the linear series is indeed $g+3$, as $\phi$ is an injective morphism.) They sit inside a very ample $g^3_{g+3}$, given by $N_p$. Again, there are $\infty^1$-many of such, as $p$ varies in $X\setminus \Ww$. Moreover, $X$ has $2g+2$ such one-dimensional families of $g^2_{g+3}$ and $g^3_{g+3}$, given by the choice $o\in \Ww$. \end{proof}

Let $\sigma: X\to X$ be the {\it hyperelliptic involution} and let $R\in \mathrm{Pic}^2(X)$ be the hyperelliptic divisor of degree two, i.e., $g^1_2 =|R| = \{a+\sigma (a)\}_{a\in X}$. 

\begin{remark}[{\bf Arbitrary characteristic}]\label{anychar}
We summarize the ingredients providing a similar proof of Theorem \ref{a1} in arbitrary characteristic. Unless otherwise stated, the statements used in the proof are valid for any algebraically closed field $k$ and any $\textnormal{char}(k)$. 

Assume $k$ to be algebraically closed and $\textnormal{char}(k) = 2$. Then $1\le \deg(\Ww) \le g+1$, where each integer in this interval may occur for some hyperelliptic curve of genus $g$ (see, e.g., \cite[7.4.24]{Liu}, \cite[\S 6.2]{Stich}). In particular, in any characteristic, there is at least one ramification point. 

Independently of $\textnormal{char}(k)$, we show that if  $p\in X\setminus \Ww$, then $h^1(\Oo _X(gp)) = 0$.
\quad Since $X$ is hyperelliptic, the canonical map $\eta : X\to \PP^{g-1}$ has as its image the degree $g-1$ rational normal
curve and its fibers are the elements of $|R|$. Thus we have the following recipe to see if an effective divisor $D$ on $X$ is special.
Let $D'\supset D$ be the following effective divisor: for each $o\in \Ww$, let $m_o$ denote the multiplicity of $o$ in $\Ww$; the
multiplicity of $o\in D'$ is the minimal even integer $\ge m_o$. If $a\in X\setminus \Ww$ and $m_1, m_2\in \NN$ are the
multiplicities of $a$ and $\sigma (a)$ in $D$, then both $a$ and $\sigma (a)$ appears in $D'$ with multiplicity $\max
\{m_1,m_2\}$. By construction, $D'$ has even degree and it is the minimal divisor containing $\eta ^{-1}(\eta (D))$, where, given $D=\sum m_ip_i$,  we set
$\eta (D):= \sum m_i \eta (p_1)$. Let $k:= \deg (D')/2$. Note that $D'\in |R^{\otimes k}|$. By Serre duality $H^1(D)\cong H^0(K_X\otimes D^{\vee})^{\vee}$ and the latter is isomorphic to  $H^0(R^{\otimes (g-1)}\otimes D^{\vee})$. Hence $h^1(D) > 0$ if and only if $k\le g-1$. In particular, let $p\in X\setminus \Ww$ and $D = \Oo _X(gp)$. Hence $D' = \Oo _X(gp + g\sigma(p))$ and so $\deg(D') = 2g$. Thus $h^1(\Oo _X(gp)) =0$. 
\end{remark}

\begin{remark}\label{a2}
Take $f(X)$ as in the proof of Theorem \ref{a1}. Call $\Ss$ the set of all $q\in \PP^3\setminus f(X)$ such that the linear
projection $\pi _q: \PP^3\setminus \{q\}\to \PP^2$ induces an injective map $\phi_q: f(X)\to  \PP^2$. Call $Q$ the quadric
surface containing $f(X)$ as an element of $|\Oo _Q(2,g+1)|$.

\quad (i) We describe the set $\Ss \cap Q$. Fix $q=(q_1,q_2)\in Q\setminus f(X)$ and set $L_1:= \PP^1\times \{q_2\}\in
|\Oo _Q(0,1)|$ and $L_2:= \{q_1\}\times \PP^1\in
|\Oo _Q(1,0)|$. Let $L \subset \PP^3$ be a line such that $q\in L$ and $\deg (L \cap f(X))\ge 2$. Since $q\notin f(X)$, B\'{e}zout's theorem 
gives $L\subset Q$. Hence $L \in \lbrace L_1, L_2\rbrace$. Thus $q\in \Ss$ if and only if both $L_1$ and $L_2$ contain a
unique point of $f(X)$.  By definition of $f$, given in the proof of Theorem \ref{a1}, $L_1$ meets $f(X)$ at a unique
point, $a_1$, if and only if $a_1 =f(p_1)$ for some $p_1\in \Ww$. Recall that $h^0(\Oo _X((g+1)p))=2$, $w = (u_1, u_2)$ where
$u_1$ is induced by the linear series $|\Oo _X((g+1)p))|$. By definition of $f$ given in the proof of Theorem \ref{a1}, $L_2$ meets $f(X)$
at a unique point, $a_2$, if and only if $a_2 =f(p_2)$ for some $p_2\in X$ such that $\Oo _X((g+1)p_2)\cong \Oo _X((g+1)p)$. The number of these points may depend on $g$, $X$ and $p$. However, there is at least one such pair of points $(p_1,p_2)\in X\times X$, i.e., the pair $(o,p)$ ($\Ww\neq \emptyset$ in any characteristic). 

\quad (ii) Fix $q\notin Q$ and take a line $L$ such that $q\in L$ and $\deg (L\cap
f(X)) \ge 2$. Since $q\notin Q$, we have $L\nsubseteq Q$. Thus $\deg (L\cap f(X))=2$, by B\'{e}zout's theorem. In particular, each line $L$ through $q$ which is tangent to $f(X)$, say at a point $q'$, has order of vanishing two with $f(X)$ at $q'$, and $L \cap (f(X)\setminus \{q'\}) =\emptyset$. Thus any unibranch point of $\pi_q(f(X))$ is an ordinary cusp.  If the degree $g+3$ curve $\pi _q(f(X))$
is unibranch, then it has $(g+2)(g+1)/2 -g$ cusps. Tono \cite[Theorem 1.1]{tono} showed that 
a cuspidal plane curve has at most $(21g+17)/2$ cusps. Thus, since $(g+2)(g+1)/2 - g > (21g+17)/2$ for $g\gg 0$, one has $\Ss \subset Q$ for $g\gg 0$. A generalization of Tono's result to Hirzebruch surfaces was found by Moe \cite{Moe}. 
\end{remark}

\begin{proposition}\label{g+3inj}
Let $X$ be a hyperelliptic curve of genus $g\ge 3$. Take any non-special and base-point free $N\in \mathrm{Pic}^{g+3}(X)$ inducing an injective map $\phi : X\to \PP^3$. Either $\phi$ is an embedding and its image $\phi(X)$ is contained in a smooth quadric $Q$ as a divisor of bidegree $(2,g+1)$ or $(g+1,2)$, or  $\phi(X)$ is contained in a quadric cone.
\begin{proof}
Since $N$ is non-special, $h^1(N)=0$. Thus Riemann-Roch gives $h^0(N)=4$.

Set $M = N\otimes R^\vee \in \mathrm{Pic}^{g+1}(X)$. Fix $a\in X$. Since $h^0(N)=4$, we have $h^0(N(-a-\sigma (a))) \ge 2$ and so $h^0(M)\ge 2$. Since $\phi$ is injective, we have $\phi(a)\ne \phi(\sigma (a))$ for $a\in X\setminus \Ww$. Thus $h^0(M) = h^0(N)-2=2$. 

Since $\deg (M)=g+1$, and $h^0(M)=2$, Riemann-Roch implies that $M$ is non-special. 

Assume that $M$ has a base point, say $b\in X$. Since $h^0(M(-b)) =2$ and $\deg (M(-b)) =g$, $M(-b)$ is a special, with $h^0(M(-b))=2$. Thus $|M(-b)| = R\otimes E$ for a fixed effective divisor $E$ with $\deg (E) =g-2> 0$. 
Note that $M \cong R(E+b)$. Since by definition $M = N\otimes R^{\vee}$, tensoring by $R$ both sides yields $N\cong R^{\otimes 2}(E+b)$. 

Note that the divisor $E$ is a fixed component for the linear series associated to $N(-b)$. Indeed, this holds if and only if $h^0(N(-b-E)) = h^0(N(-b))$. Moreover, $N(-b-E) \cong R^{\otimes 2}$ and so $h^0(N(-b-E)) = h^0(R^{\otimes 2}) = 3$. Furthermore, $h^0(N(-b)) = 3 = h^0(N)-1$, as $N$ is base-point free. Since $N(-b-E)$ is base-point free and $h^0(N(-b-E)) = 3$, this induces a map from $X$ to $\PP^2$, which factor through $\phi$ (the morphism induced by $N$). More precisely, $N(-b-E)$ induces $\pi_{\phi(b)}\circ\phi: X\rightarrow \PP^2$, where $\pi_{\phi(b)}$ is the the linear projection with center the point $\phi(b)$. On the other hand, $R^{\otimes 2}$ has $\deg(R^{\otimes 2}) = 4$ and induces a $2:1$ cover $X\rightarrow \PP^1 \subset \PP^2$, where $\PP^1\subset \PP^2$ is a smooth conic. As $\phi$ is injective, $\pi_{\phi(b)}$ is a $2:1$ cover of a smooth conic. This is possible only if $\phi(X)$ is contained in a quadric cone such that $\phi(b)$ is a vertex. 

The map $\pi_{\phi(b)}\circ \phi$ sends all the points in the support $\supp(E)$ of $E$ to $\phi(b)$ (because $E$ is a fixed component of $N(-b)$) and since $\phi$ is injective, $\supp(E)\subseteq \lbrace b \rbrace$. Since $\deg(E) = g-2$, then $E = (g-2)b$. Hence $N = R^{\otimes 2}\otimes \Oo_X((g-1)b)$. 

Conversely for a general $b\in X$, the linear series  $|R^{\otimes 2}((g-1)b)|$ is non-special by Remark \ref{anychar} and gives an injective map with image contained in a quadric cone.

Suppose $M$ is base-point free and call $\psi: X\to \PP^1$ the morphism induced by $|M|$. Since $M$ is base-point free, $h^0(M(-b)) = h^0(M)-1=1$ for every $b\in X$. Then Riemann-Roch gives $h^1(M(-b)) =0$ for every $b\in X$. As in the proof of Theorem \ref{a1}, we see that $N$ induces an embedding with image contained in a smooth quadric $Q$ as a divisor of bidegree $(2,g+1)$ or $(g+1,2)$.
\end{proof}
\end{proposition}

\begin{proposition}\label{g+3injbis}
Keep the notation from Proposition \ref{g+3inj}. Assume $\phi (X)$ is contained in a quadric cone. Then $\phi$ is an embedding if and only if $g=2$.
\begin{proof}
Recall that in this case $M=N\otimes R^{\vee}$ has $b$ as base point. 

Suppose $g=2$, then $\deg(N) = g+3 = 2g+1$ and so $N$ is very ample, and hence an embedding. Suppose $g>2$, then $E = (g-2)b$ is non-zero. Recall
that $E$ is the fixed component of $N(-b)$. Hence $h^0(N(-b)) = h^0(N(-2b))$. Thus $\phi$ is not an embedding, as $N$ does not divide tangent directions, i.e., the differential of $\phi$ is not injective at $b$. (Note that for $g=2$, $\phi(X)$ has degree $5$ \cite[Example V.2.9]{Hart}.) 

Assume $g\ge 3$ and that $\phi(X)$ is contained in a quadric cone  $\mathcal C$ with vertex $v=\phi (b)$ and take $q\in \mathcal C\setminus \phi(X)$. Here we check that the linear projection from $q$ does not induce an injective map $X\to \PP^2$. Call $R_q$ the unique line on $\mathcal C$ containing $q$. By B\'{e}zout's theorem, for each line $L$ containing $q$ and with $\deg (L\cap \phi(X))\ge 2$, we have $L\subset \mathcal C$. Recall that $v=\phi(b)$ is the vertex of $\mathcal C$. Since the vertex $\phi (b)\in\phi (X)$, the projection $\pi_q$ of $\phi(X)$ from $q$ is a cuspidal projection if and only if $R_q\cap \phi (X) =\{\phi (b)\}$ set-theoretically.  

We show this is not the case. The linear projection $\pi_v$ of $\phi(X)$ from $v$ is a $2:1$ morphism (away from $v\in \phi(X)$), whose image is a smooth conic, i.e., the base of the cone $\mathcal C$. Thus the ramification points of $\pi_v$ are
the images of the Weierstrass points $\phi(\Ww)$, the image of the ramification points of the covering map induced by $R$. However, since $b\notin \Ww$, the point $\phi(b)$ is not a ramification point. Thus $R_q$ cannot intersect $\phi (X)$ only at $\phi (b)$, i.e., $\pi_q$ is not a cuspidal projection. 
\end{proof}
\end{proposition}

Recall that $\sigma: X\to X$ denotes the hyperelliptic involution and $R\in \mathrm{Pic}^2(X)$ is the hyperelliptic divisor of degree two, i.e., $g^1_2 =|R| = \{a+\sigma (a)\}_{a\in X}$. With this notation, we are ready to prove the next result. 

\begin{proposition}\label{a3}
Let $X$ be a hyperelliptic curve of genus $g\ge 3$. There is an injective morphism $f: X\to \PP^2$ with $\deg (f(X)) =g+2$ and each such map $f$ is induced by a complete linear series $|N|$ with $h^1(N) =0$ and $N \cong R(gp)$ with $p\in X$ such that $h^1(\Oo _X(gp)) =0$ (e.g., with $p$ general in $X$).
\end{proposition}

\begin{proof}
Since every special base-point free linear series on $X$ is composed with the $g^1_2$, each injective morphism $X\to \PP^2$
must be induced by a non-special base-point free linear series $g^2_d$. By Riemann-Roch this linear series is complete if and only if $d=g+2$, whereas if $d>g+2$ this $g^2_d$ is a linear subspace of a base-point free and non-special complete $g^{d-g}_d$.
\begin{quote}
\noindent {\it Claim 1:} For every $N\in \mathrm{Pic}^{g+2}(X)$, $g\ge 3$,  with $h^1(N)=0$ and $N$ base-point free
there is a degree $g$ effective divisor $B$ with $N\cong R(B)$ and $h^1(B)=0$. \\

\noindent \emph{Proof of Claim 1:} Fix $a\in X$. Since every degree two effective divisor of $X$ is contained in a prescribed $g^2_k$, there is a degree $g$ effective
divisor $B$ such that $a+\sigma (a) +B \in |N|$. Note that $a+\sigma (a)\in g^1_2$. Hence $h^0(N\otimes R^\vee )>0$. Take
$B\in |N\otimes R^\vee|$. If $h^1(B)>0$, then $B$ is special and so $B = R\otimes N'$, where $N'$ is some effective divisor of degree $g-2$. Thus
$N \cong R^{\otimes 2}\otimes N'$. Since $\deg (N') = g-2>0$, and $h^0(R^{\otimes 2})=3 =h^0(N)$, every point in the support of a divisor of $N'$
is a base point of $N$, a contradiction.
\end{quote}

Thus, so far we have shown that $N\cong R(B)$ for some effective divisor $B$ with $\deg (B)=g$ and $h^1(B)=0$. Since by assumption $R\otimes B$ induces
a map to $\PP^2$, it is base-point free. Note that $R(B)$ is base-point free if and only if $h^0(R(B-p))
=2$ for each $p$ in the support of $B$; indeed, since $R$ is base-point free, the base locus of $R(B)$ has to be contained in the support of $B$.
Moreover, since by assumption $N$ induces an injective morphism and $h^{0}(R(B-p)) = 2 = h^{0}(R)$, $|R(B)|$ maps all the points in the support of $B$ to the same point of $\PP^2$. Therefore $B =gp$ for some $p\in X$ (and such that $h^1(\Oo _X(gp)) =0$). 

Conversely, assume $h^1(\Oo _X(gp)) =0$ and set $N:= R(gp)$. Call $\phi: X\to \PP^2$ the morphism induced by the non-special
and base-point free linear series $|N|$. We claim that $\phi$ is injective. Fix $a, b\in X$ with $a\ne b$. First
assume $\phi(a) = \phi({p})$. Thus $|R(gp-a)| = |R((g-1)p)|$, with $h^{0}(R(g-1)p) = 2 = h^{0}(R)$. Hence $(g-1)p$ is the base locus of $R((g-1)p)$ and so of $R(gp-a)$. This implies $a=p$. 

Now assume $a\ne p$ and $b\ne p$. Since $\phi(a) = \phi(b)$, by definition $b$ is a base point of $|R(gp-a)|$. Thus $|R(gp-a)| =
\{b+E\}_{E\in |R(gp-a-b)|}$. Since $\deg (R(gp-a-b)) =g$ and $h^0(R(gp-a-b))=2$, by Riemann-Roch we obtain $h^1(R(gp-a-b))>0$ and hence
$R(gp-a-b)$ is a special divisor. Thus $R(gp-a-b) = R\otimes F$, for some effective
degree $g-2$ divisor $F$ on $X$. So $gp-a-b$ is an effective divisor. This is possible if and only if $a=b=p$, which is a contradiction.
\end{proof}

\section{Quadrics and cuspidal projections}\label{quadrics}

In this section, we study more closely cuspidal projections of curves lying on irreducible quadrics in $\PP^3$. In the first results the characteristic
of our ground field $k$ is arbitrary. Only later, we will switch to characteristic zero. We start off considering curves on smooth quadrics.

\begin{proposition}\label{b1}
Let $Q\subset \PP^3$ be a smooth quadric surface. Fix integers $1\le d_1\le d_2$ such that $d_2\ge 2$. Let $Y$ be an integral
element of $|\Oo _Q(d_1,d_2)|$ with only unibranch singularities and let $\phi: X\to Y$ be its normalization. Take $q\in Q\setminus Y$ and let $L\in
|\Oo _Q(1,0)|$ and
$L'\in |\Oo _Q(0,1)|$ be the unique lines of $Q$ through $q$. The linear projection from $q$ induces an injective map $\pi_q : Y\to
\PP^2$ (and hence an injective map $\eta =\pi_q \circ \phi: X\to \PP^2$) if and only if each $L$ and $L'$ contains a unique point
of $Y$. Moreover, the morphisms $\pi_q$ and $\eta$ are separable. 
\end{proposition}

\begin{proof}
Notice that $\pi_q$ is a morphism, because $q\notin Y$. Since $\phi$ is bijective and separable, and an isomorphism outside finitely many points of $X$, $\pi_q$ is injective (resp., separable) if and only if $\eta$ has the same property. If $\eta$ is injective then $|L\cap
Y|=|L' \cap Y|=1$. To show the converse, it is sufficient to prove that $|L''\cap Y|=1$ for each line $L''\subset \PP^3$
such that $q\in L''$ and $L''\notin \{L,L'\}$.

Now we  explain why $\pi_q$ and $\eta$ are separable morphisms. As mentioned above, $\eta$ is
separable if and only if $\pi_q$ is separable. Separability must be checked only if
$\textnormal{char}(k)>0$, as it is immediate in characteristic zero. Fix
$p\in Y_ {\mathrm{reg}}$ such that $p \notin L\cup L'$; call $L_p\subset \PP^3$ the line spanned by $\{q, p\}$.
Since $p\notin L\cup L'$, we have $L_p\notin \{L,L'\}$ and hence $\deg (L_p\cap Q) = 2$. Thus the differential
of $\eta$ at $p$ is injective. This shows that the differential is generically injective, and so $\eta$ is separable. 
\end{proof}

\begin{remark}
The proof of Proposition \ref{b1} shows that if $d_1>1$ there are only finitely many cuspidal projections. (Note that if $(d_1,d_2) \ne (2,2)$, having at least one cuspidal projection is a closed condition on the smooth curves of bidegree $(d_1,d_2)$.) For curves of bidegree $(1,d_2)$, which are smooth and rational, if there is a cuspidal
projection from a point $o$, then any point on $Q\setminus Y$ and on the line of bidegree $(1,0)$ containing $o$ induces a cuspidal projection.
\end{remark}

The next result shows how zero-dimensional schemes may naturally provide information on the {\it existence} of cuspidal projections and therefore 
injective linear series. (Note that, for $\textnormal{char}(k)=0$, its proof may be simplified using the classical Bertini's theorem \cite[Corollary III.10.9]{Hart}.)

\begin{theorem}\label{b2}
Fix integers $d_2\ge d_1>0$, a smooth quadric $Q\subset \PP^3$, lines $L\in |\Oo _Q(1,0)|$, $L'\in |\Oo _Q(0,1)|$ and set $\{q\}:=
L\cap L'$. Fix
$o\in L\setminus \{q\}$ and $o'\in L'\setminus \{q\}$. Let $Z$ (resp. $Z'$) be the divisor $d_2\cdot o = o+\cdots+o \subset L$ (resp., $d_1\cdot o'\subset L'$)
regarded as a degree $d_2$ (resp., degree $d_1$) zero-dimensional scheme. Then there is a smooth divisor $Y\in |\Oo _Q(d_1,d_2)|$ such that
$Y\cap L = Z$ and $Y\cap L'=Z'$ (set-theoretically, they intersect at a unique point).
\end{theorem}

\begin{proof}
Since $h^1(\Oo _Q(d_1-1,d_2-1)) =0$, from the residual exact sequence
$$
0 \to \Oo _{Q}(d_1-1,d_2-1) \to \Ii _{Z\cup Z'}(d_1,d_2) \to \Ii _{Z \cup Z', L\cup L'}(d_1,d_2)\to 0,
$$
upon taking global sections, it follows $h^0(\Ii _{Z\cup Z'}(d_1,d_2)) = h^0(\Oo _Q(d_1,d_2)) -d_1-d_2 = d_1d_2+1$. Similarly, one can directly check that
$h^0(\Ii _{Z'}(d_1-1,d_2)) =d_1(d_2+1)-d_1$, and $h^0(\Ii _Z(d_1,d_2-1)) =(d_1+1)d_2 -d_2$. (This is the stabilization of the Hilbert function to the Hilbert polynomial.)
It is sufficient to prove that a general $Y\in |\Ii _{Z\cup Z'}(d_1,d_2)|$ is smooth. Since $h^0(\Ii _{Z'}(d_1-1,d_2)) =d_1(d_2+1)-d_1$ and $h^0(\Ii _Z(d_1,d_2-1)) =(d_1+1)d_2 -d_2$, neither $L$ nor $L'$ is an irreducible component of $Y$, by dimensional count.

First, assume $d_1>1$. Since
$|\Ii _{Z\cup Z'}(d_1,d_2)|$ contains all curves of the form $F\cup L\cup L'$, where $F\in |\Oo _Q(d_1-1,d_2-1)|$, and $\Oo _Q(d_1-1,d_2-1)$ is very ample,
the linear series $|\Ii _{Z\cup Z'}(d_1,d_2)|$ separates points and tangent vectors of $Q\setminus (L\cup L')$, i.e., it induces an
embedding
$\psi : Q\setminus (L\cup L')\to \PP^{d_1d_2}$. For any point $p\in Q$, let $2p$ denote the degree $3$ zero-dimensional subscheme of $Q$ whose ideal sheaf is $(\Ii _{p,Q})^2$. Since $|\Ii _{Z\cup Z'}(d_1,d_2)|$ is irreducible, its general element is smooth outside the locus $L\cup L'$. Therefore, in order to establish the result, it is sufficient to prove the following statements:
\begin{enumerate}
\item[(i)] $h^0(\Ii _{Z\cup Z'\cup 2o}(d_1,d_2)) \le d_1d_2$;
\item[(ii)] $h^0(\Ii _{Z\cup Z'\cup 2o'}(d_1,d_2)) \le d_1d_2$;
\item[(iii)] $h^0(\Ii _{Z\cup Z'\cup 2q}(d_1,d_2)) \le d_1d_2$;
\item[(iv)] for each $m\in L\setminus \{q,o\}$ we have $h^0(\Ii _{Z\cup Z'\cup 2m}(d_1,d_2)) \le d_1d_2 -1$;
\item[(v)] for each $m'\in L'\setminus \{q,o'\}$ we have $h^0(\Ii _{Z\cup Z'\cup 2m'}(d_1,d_2)) \le d_1d_2 -1$.
\end{enumerate}

To show (i) and (ii), fix $F\in |\Oo _Q(d_1-1,d_2-1)|$ such that $o\notin F$ and $o'\notin F$. Notice that $F\cup L\cup L' \in |\Ii _{Z\cup Z'}(d_1,d_2)|$. 
Since $L\cap {Z\cup Z'\cup 2o}$ has degree $d_2+1$, $h^{0}(I_{Z\cup Z'\cup 2o}(d_1, d_2)) = h^{0}(I_{Z\cup Z'}(d_1,d_2-1)$, as every global section in  
$I_{Z\cup Z'\cup 2o}$ contains $L$. Similarly for $L'$ with $o'$. Now, $h^{0}(I_{Z\cup Z'}(d_1,d_2-1), h^{0}(I_{Z\cup Z'}(d_1-1,d_2)\leq d_1d_2$ and so (i) and (ii) are proven. 

To show (iii), let $G\in |\Ii _{Z\cup Z'\cup 2q}(d_1,d_2)|$. Since $\deg (L\cap Z\cup Z'\cup 2q)=d_1+2$, $L$ is a component of $G$. Since $\deg (L'\cap Z\cup Z'\cup 2q)=d_1+2$, $L'$ is a component of $G$. Thus $h^0(\Ii _{Z\cup Z'\cup 2q}(d_1,d_2)) =h^0(\Oo _Q(d_1-1,d_2-1))=d_1d_2$.

For (iv), fix $m\in L\setminus \{q,o\}$. Since $(Z\cup Z' \cup 2m)\cap L$ is the union of $Z$ and the degree two effective divisor of $L$ with $m$ as its support, we have $\mathrm{Res}_L(2m\cup Z\cup Z') = \{m\}\cup Z'$. Thus we have the following residual exact sequence of $L$ in $Q$:
\begin{equation}\label{eqb1}
0 \to \Ii _{Z'\cup \{m\}}(d_1-1,d_2) \to \Ii _{Z\cup Z'\cup 2m}(d_1,d_2) \to \Ii _{(Z\cup 2m)\cap L,L}(d_1,d_2)\to 0
\end{equation}
Since $\deg (Z\cup 2m)\cap L)=d_2+2$, we have $h^0(L, \Ii_{(Z\cup 2m)\cap L,L}(d_1,d_2))=0$. Thus, taking global sections, the exact sequence of sheaves (\ref{eqb1}) gives $h^0(\Ii _{Z\cup Z'\cup 2m}(d_1,d_2)) = h^0(\Ii_{Z'\cup \{m\}}(d_1-1,d_2))=d_1(d_2+1)-d_1-1$.  Similarly, (v) is derived using the residual exact sequence of $L'$.

Now assume $d_1=1$. Take any $Y\in |\Oo _Q(1,d_2)|$ and assume that $Y$ has a singular point $z\in Y$. Let $R_z$ be the element of $|\Oo _Q(1,0)|$ passing through $z$. Since $Y$ is singular at $z$ and $d_1=1$, B\'{e}zout's theorem gives $Y =R_z\cup G$, for some $G\in |\Oo _Q(0,d_2)|$. If $Y\in |\Ii _{Z\cup Z'}(1,d_2)|$ and $R_z\cap \{o,o'\}=\emptyset$ (this condition holds for a general $Y$), we obtain $G\in |\Ii _{Z\cup Z'}(0,d_2)|$. Notice that $h^{0}(\Ii _{Z\cup Z'}(0,d_2))\leq d_2$. Thus a general $Y\in |\Ii _{Z\cup Z'}(1,d_2)|$ is smooth outside $\{o,o'\}$. As it was shown for (i) and (ii), one verifies that the general $Y$ is smooth at $o$ and $o'$.
\end{proof}

\begin{remark}\label{b3}
Let $d_1=1$, $d_2=d-1\ge 2$ and $Y$ be a smooth rational curve. By Proposition \ref{b1} and Theorem \ref{b2}, we obtain that for each integer
$d\ge 3$ there is a smooth rational curve $Y\subset \PP^3$ with $\deg (Y)=d$ and admitting $\infty ^1$ cuspidal projections to $\PP^2$. Compare this observation with the statement \cite[(a) of Remark, p. 102]{Piene} saying that, for $d\ge 5$, no smooth degree $d$ rational curve has a cuspidal projection with as its image a plane curve with only ordinary cusps.
\end{remark}

We now establish Proposition \ref{b1} and Theorem \ref{b2} for quadric cones in $\PP^3$. 
\begin{proposition}\label{b4}
Let $\mathcal C\subset \PP^3$ be a quadric cone with vertex $v$. Fix an integer $d\ge 2$. Let $Y\in |\Oo _Y(d)|$ be an integral curve
with only unibranch singularities and let $\phi: X\to Y$ be its normalization. Fix $q\in \mathcal C\setminus Y$ with $q\ne v$ and let $R_q$ be the line of $\mathcal C$
containing $q$. The linear projection from $q$ induces a cuspidal projection of $Y$ (and hence of $X$ taking the composition with the injective map $\phi$)
if and only if $|R_q\cap Y| =1$.
\end{proposition}

\begin{proof}
Take a line $L\subset \PP^3$ such that $q\in L$ with $\deg (Y\cap L)\ge 2$. Since $q\in \mathcal C\setminus Y$, we have $\deg (L\cap \mathcal C)\ge 3$. Thus $L\subset \mathcal C$, by B\'{e}zout's theorem. Since $q\in L$, $L=R_q$.
\end{proof}

\begin{theorem}\label{b4.1}
Let $\mathcal C\subset \PP^3$ be a quadric cone with vertex $v$. Fix an integer $d\ge 2$ and $q\in \mathcal C\setminus \{v\}$. Let $R_q\subset \mathcal C$ be the line spanned by $\{v,q\}$. Then there is a smooth divisor $Y\in |\Oo _{\mathcal C}(d)|$ such that $v\notin Y$, $q\notin Y$ and $R_q$ meets $Y$ at a unique point.\end{theorem}

\begin{proof}
Fix $p\in R_q\setminus \{v\}$. Let $Z = d\cdot p$ be the effective divisor of degree $d$ of $R_q$ supported at $p\in \mathcal C$, regarded as a zero-dimensional subscheme of $\mathcal C$. It is sufficient to prove that a general element of $|\Ii _Z(d)|$ is smooth and it does not contain the vertex $v$.

Let $\eta : \mathbb F_2\to \mathcal C$ be the minimal desingularization of $\mathcal C$; here $\mathbb F_2$ denotes the second {\it Hirzebruch surface}: this is the rational ruled surface $\PP(\mathcal \Oo_{\PP^1}\oplus \Oo_{\PP^1}(-2))$. Thus we have a projection map $\pi: \mathbb F_2\to \PP^1$ and a section $H:= \eta ^{-1}(v)$ with self-intersection $H^2=-2$. Its Picard group $\mathrm{Pic}(\mathbb F_2)$ is freely generated by the Cartier divisors $H$ and a fiber $F$ of $\pi$, with $F\cdot H = 1$ and $F^2 = 0$ \cite[Proposition V.2.3]{Hart}. 

The linear series $|\Oo _{\mathbb F_2}(H+2F)|$ is base-point free, induces $\eta$ and indeed contracts $H$ to a point. Hence $\eta^\ast(\Oo _{\mathcal C}(d)) = \Oo _{\mathbb F_2}(dH+2dF)$, for every $d\geq 1$. In fact, $\eta ^\ast (H^0(\Oo _{\mathcal C}(d))) = H^0(\Oo _{\mathbb F_2}(dH+2dF))$. Indeed, since $\mathcal C$ is a quadric and it is projectively normal, one has $h^0(\Oo _{\mathcal C}(d)) =h^0(\Oo _{\PP^3}(d)) -h^0(\Oo _{\PP^3}(d-2)) =(d+1)^2$. On the other hand, note that $\pi _\ast \Oo _{\mathbb F_2}(dH) \cong \textnormal{Sym}^d(\Oo _{\PP^1}\oplus \Oo _{\PP^1}(-2)) \cong \bigoplus_{i=0}^{d} \Oo _{\PP^1}(-2i)$ by \cite[Proposition V.2.8.]{Hart} and  \cite[Exercise III. 8.4]{Hart}. As $F$ is a fiber of $\pi$ over $\PP^1$, $\Oo _{\mathbb F_2}(F) \cong \pi ^\ast (\Oo _{\PP^1}(1))$. The projection formula \cite[Exercise II. 5.1]{Hart} gives the isomorphism $\pi _\ast (\Oo _{\mathbb F_2}(dH+2dF)) \cong \oplus _{i=0}^{d}\Oo _{\PP^1}(2d-2i)$. Now, the dimension of the space of global sections of the latter coincides with $h^{0}(\Oo _{\mathcal C}(d))$. Thus $\eta^\ast(H^{0}(\Oo _{\mathcal C}(d))) = H^{0}(\Oo _{\mathbb F_2}(dH+2dF))$.

Since $p\ne v$, the scheme $A:= \eta ^{-1}(Z)$ is a degree $d$ zero-dimensional scheme and $h^0(\mathcal C,\Ii _Z(d)) =h^0(\mathbb F_2,\Ii _A(dH+2dF))$. Therefore, to prove the statement it is sufficient to prove that a general $W\in |\Ii _A(dH+2dF)|$ is smooth and $H\cap W =\emptyset$. 

Let $R_A\in |\Oo _{\mathbb F_2}(F)|$ denote the element containing $A$ (i.e., it is the strict transform of $R_q$). The residual exact sequence of $R_A$ in $\mathbb F_2$ gives the exact sequence
\begin{equation}\label{eqb2}
0 \to \Oo _{\mathbb F_2}(dH+(2d-1)F)\to \Ii _A(dH+2dF) \to \Ii _{A,R_A}(dH+2dF)\to 0.
\end{equation}
With analogous computations as above, we have $\pi _\ast (\Oo _{\mathbb F_2}(dH+(2d-1)F)) \cong \oplus _{i=0}^{d} \Oo _{\PP^1}(2d-1-2i)$ 
and so $h^1(\pi _\ast (\Oo_{\mathbb F_2}(dH+(2d-1)F))=0$. By \cite[Lemma V.2.4]{Hart}, we have $H^1(\Oo_{\mathbb F_2}(dH+(2d-1)F)))\cong H^1(\PP^1,\pi _\ast (\Oo_{\mathbb F_2}(dH+(2d-1)F)) = 0$ and hence  $h^1( \Oo _{\mathbb F_2}(dH+(2d-1)F)) =0$. 

Since $R_A\cong \PP^1$ and $A$ is a zero-dimensional scheme of degree $d$, $\Oo _{R_A}(dH+2dF)$ has degree $d$ and so $h^0(R_A, \Ii_{A,R_A}(dH+2dF))=1$. Thus $W\in |\Ii _A(dH+2dF)|$ containing $R_A$ are of codimension one. Therefore, the general $W$ does not contain $R_A$, which yields that $\eta(W)$ does not intersect $R_q$ outside $Z$. 

The divisor $\Oo _{\mathbb F_2}((d-1)H+(2d-1)F)$ is very ample. Since $H\cup R_A\cup G\in |\Ii _A(dH+2dF)|$ for all $G\in |\Oo _{\mathbb F_2}(d-1)H+(2d-1)F)|$, the linear system $|\Ii _A(dH+2dF)|$ induces an embedding outside $H\cup R_A$. By a characteristic free version of Bertini's theorem for embeddings of quasi-projective varieties \cite[Th. 6.3, (3)]{j}, a general $W\in |\Ii _A(dH+2dF)|$ is smooth outside $H\cup R_A$. 
 
We only need to check that a general $W$ is smooth at $q = \eta^{-1}(p)$, the support of $A$. Since smoothness at $q$ is an open condition, it is sufficient to exhibit a $W'\in |I_A(dH+2dF)|$ that is smooth at $q$: take $W' = G\cup H\cup R_A$ with $G\in |\Oo _{F_2}((d-1)H+(2d-1)F)|$ and $q\notin G$; it is possible to choose such a $G$ as
$\Oo _{\mathbb F_2}((d-1)H+(2d-1)F)$ is very ample and in particular base-point free. 

Moreover, $H\cap W =\emptyset$ for a general $W\in |\Ii _A(dH+2dF)|$, as $h^0(\Oo _{\mathbb F_2}((d-1)H+2dF)) < h^0(\Oo _{\mathbb F_2}(dH+2dF))$ and one has zero intersection index between the two: $W\cdot H = (dH+2dF)\cdot H = dH^2+2dF\cdot H = -2d +2d = 0$.\end{proof}

\begin{example}\label{b4.1}
Let $X\subset \PP^3$ be the canonical model of a smooth and non-hyperelliptic curve of genus $4$. It is known that $X$ is the
complete intersection of an integral quadric $\mathcal C$ and a cubic surface; moreover, such $\mathcal C$ is smooth if and only if $X$ has two
different $g^1_3$'s (in this case the $g^1_3$'s are induced by the two rulings of $\mathcal C$), whereas if $\mathcal C$ is a quadric cone with
vertex $v$, then $v\notin X$ and $X$ has set-theoretically a unique $g^1_3$. This $g^1_3$ is induced by the linear
projection from $v$; see \cite[\S 4]{kempf}. The case $\mathcal C$ smooth is the one described in Proposition \ref{b1} with $X=Y$ and $d_1=d_2=3$. By Proposition \ref{b1} and Proposition \ref{b2}, we obtain that $X$ has a cuspidal projection if all the $g^1_3$ on $X$ have a total ramification point, i.e., $3p\in g^1_3$ for some $p\in X$. The converse holds if we only consider projections from points of the quadric surface containing $X$. Furthermore, \cite[Theorem 2]{Piene} gives a different and stronger result: the canonical model of a general curve of genus $4$ has no cuspidal projections.
\end{example}

Equipped with the terminology in Definition \ref{deftypes}, we may state the following result, which provides positive examples to Question \ref{q1}.

\begin{theorem}\label{pp1}
Let $k$ be an algebraically closed field of arbitrary characteristic and let $d_2\geq d_1\geq 1$. Then there exists a smooth genus $g$ curve with an injective $g^2_{d_1+d_2}$ of {\it type} {\bf II} with $g= d_1d_2-d_1-d_2+1$.  
\begin{proof}
This is a consequence of Theorem \ref{b1} and Theorem \ref{b4.1} on smooth quadrics and cones, respectively. 
\end{proof}
\end{theorem}

Henceforth, we assume $\mathrm{char}(k)=0$. 

\begin{definition}
For each germ $(Y,z)$ of an isolated one-dimensional singularity, the {\it singularity degree} $\delta (Y,z)$
is the nonnegative integer such that, given the normalization $\nu: X\to Y$, the arithmetic genus of $Y$ is 
$p_a(Y) = p_a(X)+\sum _{z\in \mathrm{Sing}(Y)}\delta(Y,z)$. 
\end{definition}

\begin{definition}
For each $h\geq 1$, a singularity $A_{2h}$ is a plane curve singularity which is formally equivalent to the singularity $y^2=x^{2h+1}$ at $(0,0)\in \AA^2$; see \cite[Example 3 at p. 549]{gls2}, \cite[Corollary 1.1.41, Lemma 1.1.78]{GLS}. 

Its {\it singularity scheme} is any degree $3h+2$ connected zero-dimensional $Z\subset \AA^2$ isomorphic to the subscheme of $\AA^2$ with ideal 
$I_Z = (y^2,yx^{h+1}, x^{2h+1})$. An $A_{2h}$-{\it singularity scheme} is any connected subscheme of a smooth surface isomorphic to the singularity scheme $Z$. The singularity degree of an $A_{2h}$-singularity is $h$. We use the convention that the $A_0$-singularity scheme is a smooth point. 
\end{definition}

We use the following theorem proved by Ro\'{e} \cite[Theorem 1.2]{roe2}: 

\begin{theorem}[{\bf Ro\'{e}}]\label{bb1}
Fix positive integers $h$ and $t\geq 13$. If $3h+2\le \binom{t+2}{2}$, then there is an $A_{2h}$-singularity scheme
$Z\subset \PP^2$ such that $h^1(\PP^2,\Ii_Z(t)) =0$. 
\end{theorem}

\begin{lemma}\label{bb2}
Fix positive integers $h$ and $t\geq 14$. If $3h+2\le \binom{t+1}{2}$, then there exist two $A_{2h}$-singularity schemes $A, B\subset Q$
such that $h^1(Q,\Ii_{A\cup B}(t,t)) =0$.
\end{lemma}

\begin{proof}
Fix $o\in \PP^3\setminus Q$. Let $\pi_o: \PP^3\setminus \{o\}\to \PP^2$ denote the linear projection from $o$. Since $o\notin
Q$, $\pi _{o|Q}$ defines a degree $2$ finite morphism $\pi: Q\to \PP^2$. The ramification locus $R\subset Q$ of $\pi$ is
a smooth conic which is the intersection of $Q$ with the plane polar to $o$ with respect to $Q$. The branch locus $\pi
(R)\subset \PP^2$ is a smooth conic. By Ro\'{e}'s Theorem \ref{bb1}, there is an $A_{2h}$-singularity scheme $Z\subset \PP^2$ such
that $h^1(\PP^2,\Ii _Z(t-1))=0$. Applying an automorphism to $\PP^2$, we may assume $Z\cap f(R)=\emptyset$. Thus $\pi^{-1}(Z)$ is the disjoint union of the $A_{2h}$-singularity schemes, say $A$ and $B$. Since $\pi$ is a degree $2$ covering between smooth varieties and the branch locus of $\pi$ is a conic, $\pi _\ast (\Oo_Q)\cong \Oo_{\PP^2}\oplus \Oo_{\PP^2}(-1)$. (This can be checked on a local chart.) Since $\pi$ is a finite morphism, $R^i\pi_\ast(\Ff)=0$ for all $i>0$ and any coherent sheaf $\Ff$ on $Q$. So $\pi_{\ast}$ induces an isomorphism of all cohomology groups. Since
$\pi^\ast (\Ii _Z) = \Ii_{A\cup B}$, the projection formula gives $\pi _\ast (\Ii_{A\cup B}(t,t))\cong \Ii_Z(t)\oplus
\Ii_Z(t-1)$. Thus $h^1(\Ii_{A\cup B}(t,t))=h^1(\PP^2,\Ii _Z(t))+h^1(\PP^2,\Ii_Z(t-1)) =0$. 
\end{proof}

\begin{lemma}\label{bb3}
Fix positive integers $h$ and $d_2\ge d_1\geq 15$. If $3h+2\le \binom{d_1}{2}$, then there are two $A_{2h}$-singularity schemes $A,
B\subset Q$ such that $h^1(Q,\Ii_{A\cup B}(d_1-1,d_2-1)) =h^1(Q,\Ii_{A\cup B}(d_1,d_2)) =0$ and a general $Y\in |\Ii_{A\cup B}(d_1,d_2)|$
is irreducible and with exactly two singular points, $A_{\mathrm{red}}$ and $B_{\mathrm{red}}$, both of them $A_{2h}$-singularities. 
\end{lemma}

\begin{proof}Take $A$ and $B$ as in Lemma \ref{bb2} for the integer $t=d_1-1$. Set $\{q_A\} = A_{\mathrm{red}}$ and $\{q_B\}:=
B_{\mathrm{red}}$. Lemma \ref{bb2} shows that 
$h^1(Q,\Ii _{A\cup B}(d_1-1,d_1-1))=0$. The Castelnuovo-Mumford's lemma for zero-dimensional schemes gives $h^1(Q,\Ii _{A\cup B}(d_1,d_2))=0$ and that $\Ii_{A\cup B}(d_1,d_2)$ is globally generated. 

Fix a general $D\in |\Ii_{A\cup B}(d_1,d_2)|$. Since $\Ii_{A\cup B}(d_1,d_2)$ is globally generated
and $D$ is general, $D$ has an $A_{2h}$-singularity at both $q_A$ and $q_B$ \cite[Lemma 1.1.33]{GLS}. Since $\Ii_{A\cup
B}(d_1,d_2)$ is globally generated, $D$ is smooth outside $\{q_A, q_B\}$ by Bertini's theorem in characteristic zero \cite[Corollary III.10.9]{Hart}. Thus, to conclude, it is sufficient to prove that $D$ is irreducible. Since $D$ has only finitely many singular points, $D$ has no multiple components. Since
$d_2\ge d_1\geq 2$, $\Oo_Q(d_1,d_2)$ is very ample and hence $D$ is connected. Thus if $D$ were reducible, there would be irreducible component $D'$ and $D''$ of $D$, $D'\ne D''$ and passing both through $q_A$ or $q_B$. This is a contradiction, because $D$ has unibranch singularity 
$A_{2h}$, at $q_A$ and $q_B$. 
\end{proof}

\begin{lemma}\label{bb4}
Fix integers $t_2\ge t_1 \ge 16$ and $h>0$ such that $3h+2 \le \binom{t_1-1}{2}$. Fix
$L\in |\Oo_Q(1,0)|$, $R\in |\Oo_Q(0,1)|$, $o_1\in L\setminus L\cap R$ and $o_2\in R\setminus R\cap L$. Let $Z_1$ be the degree $t_2$ divisor of $L$ supported at $o_1$. Let  $Z_2$ be the degree $t_1$ divisor of $R$ supported at $o_2$. There are two $A_{2h}$-singularity schemes
$A, B\subset Q\setminus (L\cup R)$ such that $h^1(Q,\Ii_{Z_1\cup Z_2\cup A\cup B}(t_1, t_2)) =0$ and a
general $Y\in |\Ii_{Z_1\cup Z_2\cup A\cup B}(t_1, t_2)|$ is irreducible with $\mathrm{Sing}(Y)=\{q_A,q_B\}$ and $Y$ has
$A_{2h}$-singularities at $\{q_A\} = A_{\mathrm{red}}$ and at $\{q_B\} = B_{\mathrm{red}}$.  
\end{lemma}

\begin{proof} 
We split the proof into two claims. 
\begin{quote}
\noindent {\it Claim 1}: $h^1(Q,\Ii _{Z_1\cup Z_2\cup A\cup B}(t_1, t_2))=0$, $h^1( Q,\Ii _{Z_2\cup A\cup B}(t_1-1,t_2))=0$, 
and  $h^1( Q,\Ii _{Z_1\cup A\cup B}(t_1,t_2-1))=0$. \\

\noindent {\it Proof of Claim 1}:  Since $L\cap \{o_2,q_A, q_B\}=\emptyset$, the residual exact sequence with respect to $L$ gives the exact sequence
\begin{equation}\label{eqbb1}
0 \to \Ii _{Z_2\cup A\cup B}(t_1-1,t_2)\to \Ii _{Z_1\cup Z_2\cup A\cup B}(t_1,t_2) \to \Ii_{Z_1,L}(t_1,t_2)\to 0. 
\end{equation}
Since $\deg (Z_1)=t_2$ and $\Oo_L(t_1,t_2)$ is the degree $t_2$ line bundle on $L\cong \PP^1$, we have $h^1(L,\Ii_{Z_1,L}(t_1,t_2))=0$. Thus, by \eqref{eqbb1}, to prove the assertion, it is sufficient to show the vanishing $h^1(Q,\Ii _{Z_2\cup A\cup B}(t_1-1,t_2))=0$. Since $R\cap \{q_A,q_B\}=\emptyset$, the residual exact sequence with respect to $R$ gives the exact sequence 
\begin{equation}\label{eqbb2}
0 \to \Ii _{A\cup B}(t_1-1,t_2-1)\to \Ii _{Z_2\cup A\cup B}(t_1-1,t_2) \to \Ii_{Z_2,R}(t_1-1,t_2)\to 0. 
\end{equation}
Now, Lemma \ref{bb3}, with the choice $d_1=t_1-1$ and $d_2=t_2-1$, gives $h^1(Q,\Ii _{A\cup B}(t_1-1,t_2-1))=0$ . Since $\deg (Z_2)=  t_1$ and $\Oo_R(t_1-1,t_2)$ is the degree $t_1-1$ line bundle on $R$, we have $h^1(R,\Ii_{Z_2,R}(t_1-1,t_2))=0$. The long cohomology exact sequence of \eqref{eqbb2} concludes the proof of the claim. 
\end{quote}

\begin{quote}
\noindent {\it Claim 2}: A general $Y\in |\Ii _{Z_1\cup Z_2\cup A\cup B}(t_1,t_2)|$ is smooth outside $\lbrace q_A, q_B\rbrace$, has
$A_{2h}$-singularities at $\{q_A\} = A_{\mathrm{red}}$ and at $\{q_B\} = B_{\mathrm{red}}$, and it is irreducible. \\

\noindent {\it Proof of Claim 2}:  
By Lemma \ref{bb3}, the general $D\in |\Ii _{A\cup B}(t_1-1,t_2-1)|$ is smooth outside $\{q_A,q_B\}$. Using the action of $\mathrm{Aut}(\PP^1)\times \mathrm{Aut}(\PP^1)$, we may assume that $L$ and $R$ are transversal to $D$. With this assumption, the curve $Y' = D\cup L\cup R\in |\Ii _{Z_1\cup Z_2\cup A\cup B}(t_1,t_2)|$ is smooth at $o_1$ and $o_2$. Since $\{o_1,o_2\}$ is a finite set and smoothness is an open condition, a general $Y\in |\Ii _{Z_1\cup Z_2\cup A\cup B}(t_1,t_2)|$ is smooth at $o_1$ and $o_2$. 

By {\it Claim 1}, we have $h^1(Q,\Ii _{Z_1\cup Z_2\cup A\cup B}(t_1, t_2))=0$. Thus 
\[
h^0(Q,\Ii _{Z_1\cup Z_2\cup A\cup B}(t_1, t_2))=(t_1+1)(t_2+1)-\deg (A)-\deg (B) - t_1- t_2 = 
\]
\[
= t_1t_2 - \deg(A) - \deg(B) +1. 
\]
Since $h^1( Q,\Ii _{Z_2\cup A\cup B}(t_1-1,t_2))=0$, we have 
\[
h^0( Q,\Ii _{Z_2\cup A\cup B}(t_1-1,t_2))= t_1(t_2+1) -\deg (A)-\deg (B) - t_1 = 
\]
\[
t_1t_2 - \deg(A) - \deg(B) < t_1t_2 - \deg(A) - \deg(B) +1 = h^0(Q,\Ii _{Z_1\cup Z_2\cup A\cup B}(t_1, t_2).
\]
Thus $L$ is not in the base locus of $|\Ii _{Z_1\cup Z_2\cup A\cup B}(t_1,t_2)|$. 

Since  $h^1( Q,\Ii _{Z_1\cup A\cup B}(t_1,t_2 -1))=0$, one has
\[
h^0( Q,\Ii _{Z_1\cup A\cup B}(t_1,t_2-1))= (t_1+1)t_2 -\deg (A)-\deg (B) -t_2 = 
\]
\[
t_1t_2 - \deg(A) - \deg(B) < t_1t_2 - \deg(A) - \deg(B) +1 = h^0(Q,\Ii _{Z_1\cup Z_2\cup A\cup B}(t_1, t_2)).
\]
Hence $R$ is not in the base locus of $|\Ii _{Z_1\cup Z_2\cup A\cup B}(t_1,t_2)|$. 

The base locus of $|\Ii _{Z_1\cup Z_2\cup A\cup B}(t_1,t_2)|$ is then strictly contained in $L\cup R\cup \{q_A,q_B\}$. Take any $Y\in |\Ii _{Z_1\cup Z_2\cup A\cup B}(t_1,t_2)|$ smooth at $o_1$ and $o_2$. Since $Y\cdot \Oo_Q(1,0) = t_2$, this curve meets $L$ only at $o_1$. Similarly, $Y$ meets $R$ only at $o_2$. Thus $Y$ is smooth along  $L\cup R$. By Bertini's theorem in characteristic zero, $Y$ is smooth outside $L\cup R\cup \{q_A, q_B\}$. From what we have just proven, 
a general $Y$ is smooth outside $\{q_A, q_B\}$. 

Recall that \cite[Lemma 1.1.33]{GLS} shows that the general element in $|\Ii_{ A\cup B}(t_1,t_2)|$ has 
$A_{2h}$-singularities at $q_A$ and $q_B$. Hence the set of all such curves contains a dense open subset $U$. Therefore $|\Ii_{Z_1\cup Z_2\cup A\cup B}(t_1,t_2)|\cap U\subset |\Ii_{ A\cup B}(t_1,t_2)|$ is open inside $|\Ii_{Z_1\cup Z_2\cup A\cup B}(t_1,t_2)|$; this means that the general $Y\in |\Ii_{Z_1\cup Z_2\cup A\cup B}(t_1,t_2)|$ has $A_{2h}$-singularities at $q_A$ and at $q_B$. Since $Y$ is connected and it has only unibranch singularities, $Y$ is irreducible, sa we argued in the proof of Lemma \ref{bb3}.\\
\end{quote}
\end{proof}

\begin{lemma}\label{bb5}
Fix integers $t_2\ge t_1 \ge 16$ and $h>0$ such that $3h+2 \le \binom{t_1-1}{2}$. Fix
$L\in |\Oo_Q(1,0)|$, $R\in |\Oo_Q(0,1)|$, $o_1\in L\setminus L\cap R$ and $o_2\in R\setminus R\cap L$. Let $Z_1$ be the degree $t_2$ divisor of $L$ supported at $o_1$. Let  $Z_2$ be the degree $t_1$ divisor of $R$ supported at $o_2$. There are an $A_{2\alpha}$-singularity
scheme $A'$ and an $A_{2\beta}$-singularity scheme $B'$ contained in $Q\setminus (L\cup R)$ such that $h^1(Q,\Ii_{Z_1\cup Z_2\cup A'\cup
B'}(t_1, t_2)) =0$ and a general
$Y\in |\Ii_{Z_1\cup Z_2\cup A'\cup B'}(t_1,t_2)|$ is irreducible with $\mathrm{Sing}(Y)=\{q_{A'},q_{B'}\}$ and $Y$ has
$A_{2\alpha}$-singularity at $\{q_{A'}\} = A'_{\mathrm{red}}$ and and $A_{2\beta}$-singularity at $\{q_{B'}\} =B'_{\mathrm{red}}$.
\end{lemma}

\begin{proof}
Take $A$ and $B$ as in Lemma \ref{bb4} and fix an $A_{2\alpha}$-singularity scheme $A'\subseteq A$ and an $A_{2\beta}$
singularity scheme $B'\subseteq B$. Since $A'\subseteq A$ and $B'\subseteq B,$ all vanishing occurring in the proof of Lemma \ref{bb3} holds 
true as well. Indeed, for $A'\subseteq A$ we have the exact sequence: 
\[
0\longrightarrow \Ii_{A}\longrightarrow \Ii_{A'} \longrightarrow  \Ii_{A'}/\Ii_{A} \longrightarrow 0. 
\]
Taking the long exact sequence in cohomology, for each $i>0$, we have 
\[
\cdots \longrightarrow H^i( \Ii_{A})\longrightarrow H^i( \Ii_{A'})\longrightarrow H^i( \Ii_{A'}/\Ii_{A}) = 0, 
\]
because the sheaf $\Ii_{A'}/\Ii_{A}$ is supported on a zero-dimensional scheme. So the vanishing of the left-most implies the vanishing of the middle cohomology group. Then, as in the proof of Lemma \ref{bb4}, we reach the same conclusion for any scheme $A'\cup B'\subset A\cup B$ as above. 
\end{proof}

\begin{theorem}\label{bb6}
Let $k$ be an algebraically closed field of characteristic zero. Fix integers $d_2\ge d_1\ge 16$ and $h>0$ such that $3h+2\le \binom{d_1-1}{2}$. Fix an integer $\kappa$ such that $0< \kappa \le 2h$ and set $g = d_1d_2-d_1-d_2+1-\kappa$. Then there exists a smooth genus $g$ curve with an injective $g^2_{d_1+d_2}$ of type {\bf III}. 
\end{theorem}

\begin{proof}
Fix $L\in |\Oo_Q(1,0)|$, $R\in |\Oo_Q(0,1)|$, $o_1\in L\setminus L\cap R$ and $o_2\in R\setminus R\cap L$. Let $Z_1$ be the degree $d_2$ divisor of $L$ supported at $o_1$. Let  $Z_2$ be the degree $d_1$ divisor of $R$ supported at $o_2$. 
Any $Y\in |\Oo_Q(d_1,d_2)|$ has arithmetic genus $p_a(Y) = d_1d_2-d_1-d_2+1$. By an application of Proposition \ref{b1}, it is sufficient 
to prove the existence of $\{q_A,q_B\} \subset Q\setminus (L\cup R)$ and an
integral $Y\in |\Oo_Q(d_1,d_2)|$ with only unibranch singularities $\mathrm{Sing}(Y)\subseteq \{q_A,q_B\}$ and with singularity degree $\delta (Y,q_A)+\delta
(Y,q_B) = \kappa$. (Recall that, for any $A_{2m}$-singularity, its singularity degree is $m$.) So it is enough to fix two positive integers $0< \beta \le \alpha \le h$ such that $\alpha+\beta =\kappa$ and find a $Y$ that has an $A_{2\alpha}$-singularity at $q_A$ and an
$A_{2\beta}$-singularity at $q_B$. This is the content of Lemma \ref{bb5}.  
\end{proof}

The case $\kappa = 0$ is covered by Theorem \ref{pp1}. Theorem \ref{bb6} provides positive examples to Question \ref{q1}. 

\begin{remark}
Piene \cite{Piene} considered only injective linear series  $g^2_d$ of {\it type} {\bf II}. In $\textnormal{char}(k)=0$, Proposition \ref{b1} and Proposition \ref{b4} provide curves $X$ possessing an injective (and separable) non-complete $g^2_d$ which cannot occur from \cite{Piene}, since in Piene's setting the smooth
curve $X$ is required to be embedded in $\PP^3$. So the same observation applies to the curves arising from Theorem \ref{bb6}. 
\end{remark}

\section{Inner cuspidal projections}\label{innercusp}
In this section, we work in $\mathrm{char}(k)=0$. We study inner cuspidal projections and introduce two natural sets $\mathcal A$ and $\mathcal B$ attached to them. 

Recall that for $p\in \PP^3$, let $\pi _p: \PP^3\setminus \{o\}\to \PP^2$ denote the linear projection from $p$. We look at inner {\it smooth} projections, i.e. projection from smooth points of a curve. (We only allow projection from smooth points of the
curve, because projecting from a singular point of the curve is more complicated and depends on the germ of the
singularity.) 

Let $X\subset \PP^n$ be an integral and non-degenerate curve with only cuspidal singularities. For any $o\in X_{\reg}$ the
restriction ${\pi _o}_{|X\setminus \{o\}}$ extends to a unique morphism $\pi_o: X\to \PP^2$. Define: 
\[
\mathcal A = \left\lbrace o\in X_{\reg}\mid {\pi _o}_{|X\setminus \{o\}} \mbox{ is injective} \right\rbrace
\]
and
\[
\mathcal B = \left\lbrace o\in X_{\reg}\mid \pi_o: X\to \PP^2\mbox{ is injective} \right\rbrace.
\] 

It is clear that $\mathcal B\subseteq \mathcal A$ and sometimes $\mathcal A\ne \mathcal B$, see Proposition \ref{f1}, Theorem \ref{f2}, and Remark \ref{f1.1}.  

\begin{proposition}\label{f1}
Le $Q\subset \PP^3$ be a smooth quadric surface. For any $o\in Q$, let $L_1(o)$ (resp. $L_2(o)$) denote the element of
$|\Oo_Q(1,0)|$ (resp. $|\Oo_Q(0,1)|$) containing $o$. Let $X\in |\Oo_Q(d_1,d_2)|$, $d_1>0, d_2>0$, $(d_1,d_2)\ne (1,1)$, be an integral
curve with only cuspidal singularities. Fix $o\in X_{\reg}$. Then: 
\begin{enumerate}
\item[(i)] $o\in \mathcal A$ if and only if $\#(L_1(o)\cap X)_{\red}\le 2$ and $\#(L_2(o)\cap X)_{\red}\le 2$; 
\item[(ii)] $o\in \mathcal B$ if and only if $\#(L_1(o)\cap X)_{\red}= 1$ and $\#(L_2(o)\cap X)_{\red}= 1$. This is equivalent to $L_1(o)$ (resp. $L_2(o)$) and $X$ having intersection multiplicity $b$ (resp. $a$) at $o$. 
\end{enumerate}
\end{proposition}
\begin{proof}
The assumptions on $d_1, d_2$ imply $X$ is non-degenerate. Let $L\subset \PP^3$ be a line
containing $o$ and different from $L_1(o)$ and $L_2(o)$. Since $o\in Q$ and $L\nsubseteq Q$, B\'{e}zout's theorem implies 
$\#(X\cap L)_{\red}\le 1$. If $\#(X\cap L)_{\red} = 1$, then $L$ is not tangent to $X$ at $o$. Thus to check whether $\pi_o$ is injective, it is sufficient to check the sets $\#(L_1(o)\cap X)_{\red}$ and $(L_2(o)\cap X)_{\red}$. We have $\#(L_1(o)\cap X)_{\red}=1$ (resp. $\#(L_2(o)\cap X)_{\red}=
1)$ if and only if $L_1(o)$ and $X$ have intersection multiplicity $d_2$ at $o$ (resp. $L_2(o)$ and $X$ have
intersection multiplicity $d_1$ at $o$).
\end{proof}

\begin{remark}\label{f1.1}
Note that if $L_1(o)$ and $X$ have intersection multiplicity at least $d_2-1$ at $o$ and $L_1(o)$
and $X$ have intersection multiplicity at least $d_1-1$ at $o$, then $\#(L_1(o)\cap X)_{\red}\le 1$ and $\#(L_2(o)\cap X)_{\red}\le
1$ and hence $o\in \mathcal B$ by Proposition \ref{f1}(ii).
\end{remark}

\begin{theorem}\label{f2}
Let $\mathcal C\subset \PP^3$ be a quadric cone with vertex $v$. Let $X\subset \mathcal C$ be an integral and non-degenerate curve with only
cuspidal singularities. Fix $o\in X_{\reg}$ such that $o\ne v$ and let $L_o$ be the line contained in $\mathcal C$ and passing
through $o$. The following statements hold: 
\begin{enumerate}
\item[(i)] $o\in \mathcal B$ if and only if $\#(L_o\cap X)_{\red} =1$.

\item[(ii)] $o\in \mathcal A$ if and only if $\#(L_o\cap X)_{\red} \le 2$.

\item[(iii)] Assume $v\in X_{\reg}$. Then: 
\[
v\in \mathcal A \Leftrightarrow v\in \mathcal B \Leftrightarrow X \mbox{ is a rational
normal curve}.
\]
\end{enumerate}
\end{theorem}

\begin{proof}
The first two statements are verified as the corresponding ones in Proposition \ref{f1.1}. As in the proof of Theorem \ref{b4.1}, let $\mathbb F_2$ denote the second
Hirzebruch surface. Let $\eta: \mathbb F_2 \to \mathcal C$ be the minimal
desingularization of the quadric cone $\mathcal C$, and let $Y\subset \mathbb F_2$ be the strict transform of $X$ in $\mathbb F_2$. Then $Y\in |\Oo_{\mathbb F_2}(aH+bF)|$ for some $a>0$ and $b\geq 2a$. The assumption that
$X$ is smooth at $v$ is equivalent to $b =2a+1$. 

We show $v\in \mathcal B$ if and only if $X$ is a rational
normal curve. If $X$ is a rational normal curve, one has $\mathcal B=X$ and so $v\in \mathcal B$. Conversely, assume $v\in \mathcal B$. Then the map ${\pi _v}_{|X}$ is birational onto its image, which happens if and only if $a=1$. Indeed, since the strict transform of $X$ is $Y\in |\Oo_{\mathbb F_2}(aH+bF)|$, the degree of ${\pi _v}_{|X}$
coincides with the intersection number of $Y$ with the ruling $F$, i.e. $\deg({\pi _v}_{|X}) = (aH+bF)\cdot F = a$. Since $a=1$, one has $b=3$ and hence $Y$ is smooth of genus zero. So $X$ is a rational normal curve. 

Finally, by definition $v\in \mathcal B$ implies $v\in \mathcal A$. Suppose $v\in \mathcal A$, then again ${\pi _v}_{|X}$ is birational onto its image, so $X$ is a rational normal curve 
and hence $v\in X = \mathcal B$. This completes the proof of statement (iii). 
\end{proof}

\begin{theorem}\label{f3}
Let $X\subset \PP^3$ be an integral and non-degenerate curve with only cuspidal 
singularities. The following conditions are equivalent:
\begin{enumerate}
\item[(i)] $\mathcal B$ is infinite;

\item[(ii)] $\mathcal B=X_{\reg}$;

\item[(iii)] $\deg (X)\in \{3,4\}$ and if $\deg (X)=4$, then $p_a(X)=1$. 
\end{enumerate}
\end{theorem}

\begin{proof}
Let $d = \deg (X)$. If $d\le 4$, then $X$ is contained in a quadric surface. Proposition \ref{f1} and
Theorem \ref{f2} show that (iii) implies (ii). It is clear that (ii) implies (i). 

We show that (i) implies (iii). For any $o\in \mathcal A$, the map $\pi _o$ has degree one and hence
$\deg (\pi _o(X)) =d-1$. Thus $p _a(\pi_o(X)) = (d-2)(d-3)/2$.
By assumption, $\mathcal A$ is infinite. Since $o\in X_{\reg}$ and $\pi _{o|X}$ is birational onto its image, we have
$\deg (\pi _o(X))=d-1$. Since $\mathrm{Sing}(X)$ is finite and each singular point of $X$ has Zariski tangent space of
dimension two, for infinitely many $o\in \mathcal B$, the morphism $\pi_o: X\to \pi_o(X)$ is a local isomorphism at each singular point of $X$.
Since $\pi _o$ is injective, we have $\pi _o(X_{\reg})\cap \pi _o(\Sing(X)) =\emptyset$.   Since
$\mathrm{char}(k)=0$, at all points of
$X_{\reg}$, except finitely many, the order of contact of $T_oX$ with $X$ is two. Hence, for infinitely many $o\in \mathcal B$, 
$\pi_o(o)$ is a smooth point of $\pi _o(X)$. 

By Zariski's Main Theorem, the morphism $\pi _o: X\to \pi _o(X)$ is then an isomorphism
for some $o\in X_{\reg}$. Thus $p_a(X) = p_a(\pi_o(X)) =(d-2)(d-3)/2$. 

Recall from Remark \ref{halphencastelnuovo} that Castelnuovo's bound gives $\pi(d,3) = m(m-1) + m\epsilon$, where $\epsilon\in \lbrace 0,1\rbrace$ and $d=2m+1+\epsilon$ and $m>0$.  
So 
\[
p_a(X) =\frac{(2m+\epsilon -1)(2m+\epsilon -2)}{2} = 
\begin{cases}
(2m-1)(m-1), \mbox{ for }  \epsilon=0, \\
m(2m-1), \mbox{ for } \epsilon = 1.\\
\end{cases}
\]
The inequalities $(2m-1)(m-1)\leq \pi(d,3)$ and $m(2m-1)\leq \pi(d,3)$ both imply $m=1$ and so $d\leq 4$. Therefore, $X$ is a rational normal curve if $d=3$, or $p_a(X)= 1$ if $d=4$. 
\end{proof}

\begin{remark}
In particular, $X$ is smooth with $\mathcal A=X$ if and only if $X$ is either a rational normal curve or a linearly normal elliptic curve.
\end{remark}

\bibliographystyle{amsplain}

\end{document}